\documentclass{article}
\usepackage[utf8]{inputenc}
\usepackage[english]{babel}
\usepackage[margin=1in]{geometry}
\usepackage{stix}

\usepackage{blindtext}
\usepackage{amssymb}
\usepackage{hyperref}
\hypersetup{
    colorlinks=true,
    linkcolor=blue,
    filecolor=magenta,      
    urlcolor=cyan,
}
 
\usepackage{amsthm}
\usepackage{amsmath}
\newtheorem{theorem}{Theorem}[section]
\newtheorem{corollary}[theorem]{Corollary}
\newtheorem{lemma}[theorem]{Lemma}

\theoremstyle{remark}
\newtheorem{remark}[theorem]{Remark}

\numberwithin{equation}{section}

\theoremstyle{definition}
\newtheorem{definition}{Definition}[section]

\begin{document}

\title{ACF-Monotonicity Formula on $RCD(0,N)$ Metric Measure Cones}
\author{
	Lin Sitan
	\footnotemark[1]
}
\renewcommand{\thefootnote}{\fnsymbol{footnote}}
\footnotetext[1]{E-mail address: linst3@mail3.sysu.edu.cn\newline Institute of Natural Science, Shanghai Jiao Tong University}
\date{\today}

\maketitle
\begin{abstract}
    The ACF-monotonicity formula is a powerful tool in the study of two-phase free boundary problems, which was introduced by Alt, Caffarelli, and Friedman\cite{Alt-Caffarelli-Friedman_1984}. In this paper, we extend it to $RCD(0,N)$ metric measure cones. As an application, we give a rigidity result for $RCD(0,N)$ metric measure cones.
\end{abstract}

\section{Introduction}
    In the seminal paper \cite{Alt-Caffarelli-Friedman_1984}, Alt, Caffarelli, and Friedman introduced a powerful tool, which is nowadays called \emph{ACF-monotonicity formula}. It states that: if $u_{1},u_{2}$ are continuous functions in the unit ball $B_{1} \subset \mathbb{R}^{n}$ satisfying that
    \begin{equation}
    	u_{i}(0)=0,\quad u_{i} \ge 0,\quad \Delta u_{i} \ge 0,\quad u_{1}u_{2}=0,\quad \text{in } B_{1},
    \end{equation}
    then the quantity
    \begin{equation}
    	J(r):=\frac{1}{r^{4}} \int_{B_{1}} \frac{\lvert \nabla u_{1} \rvert^{2}}{\lvert x \rvert^{n-2}} \mathrm{d}x \int_{B_{1}} \frac{\lvert \nabla u_{2} \rvert^{2}}{\lvert x \rvert^{n-2}} \mathrm{d}x
    \end{equation}
    is monotone increasing in $r \in (0,1)$. This formula has been of fundamental importance in the study of free boundary problems, especially those problems with two phases. Since then, several extensions of the ACF-monotonicity formulas has been established, including:
    \begin{itemize}
    	\item[(1)] the parabolic counterpart of the monotonicity formula \cite{Caffarelli_1993}; 
    	\item[(2)] the monotonicity formulas for more general elliptic and parabolic operators \cite{Caffarelli_1988,Caffarelli-Kenig_1998}; 
    	\item[(3)] the almost monotonicity formula for the case of $\Delta u_{i} \ge -1$ \cite{Caffarelli-Jerison-Kenig_2002} and its parabolic counterpart \cite{Edquist-Petrosyan_2008};
    	\item[(4)] the almost monotonicity formula for general elliptic and parabolic operators \cite{Matevosyan-Petrosyan_2011};
    	\item[(5)] the monotonicity formulas on Riemannian manifolds \cite{Teixeira-Zhang_2011_1,Teixeira-Zhang_2011_2}.
    \end{itemize}
    All of the works mentioned above deal with problems in smooth settings. In this paper, we will extend the ACF-monotonicity formula to non-smooth settings. More precisely, we will establish the ACF-monotonicity formula on $RCD(0,N)$ metric measure cones. The main result of this paper is the following theorem:
    \begin{theorem}\label{main_theorem}
    	Let $N \ge 2$ be an integer and let $(C(\Sigma),d_{C},m_{C})$ be the metric measure cone over an $RCD(N-2,N-1)$ metric measure space $(\Sigma,d_{\Sigma},m_{\Sigma})$ with vertex $p$. Let $u_{1},u_{2} : C(\Sigma) \to [0,\infty)$ be continuous and in $W^{1,2}_{\mathrm{loc}}(C(\Sigma))$ satisfying that:
    	\begin{itemize}
    		\item[(1)] $u_{1}(p) = u_{2}(p) = 0$;
    		\item[(2)] $u_{1} u_{2} = 0$ in $C(\Sigma)$;
    		\item[(3)] $\mathbf{\Delta} u_{i} \ge 0$ in $C(\Sigma)$, $i=1,2$.
    	\end{itemize}
        Then the quantity
        \begin{equation}
        	J(r) := \frac{1}{r^{4}} \int_{B_{r}(p)} \frac{\lvert \nabla u_{1} \rvert^{2}}{d(x,p)^{N-2}} \mathrm{d}m_{C} \int_{B_{r}(p)} \frac{\lvert \nabla u_{2} \rvert^{2}}{d(x,p)^{N-2}} \mathrm{d}m_{C}
        \end{equation}
        is monotone nondecreasing in $r \in (0,\infty)$.
    \end{theorem}
    Such metric measure cone $(C(\Sigma),d_{C},m_{C})$ can be obtained by blowing up a non-collapsed $RCD(K,N)$ metric measure space at any point. $RCD(K,N)$ metric measure spaces is the class of those metric measure spaces satisfying that Ricci curvature $\ge K$ and dimension $\le N$ in the synthetic sense. Examples of $RCD(K,N)$ metric measure spaces include Euclidean spaces, Riemannian manifolds, finite dimensional Alexandrov spaces, Ricci limit spaces.
    
    Our proof of Theorem \ref{main_theorem} relies on the following two ingredients:
    \begin{itemize}
    	\item the Stokes formulas on $RCD(K,N)$-spaces established in Section \ref{subsection_Stokes_formula}. It allows us to transform the problem to the eigenvalue estimate on the $RCD(N-2,N-1)$ metric measure space $(\Sigma,d_{\Sigma},m_{\Sigma})$;
    	\item the Polya-Szego inequality for $RCD(K,N)$ metric measure spaces \cite{Mondino-Semola_2020}. In allows us to extend the classical Friedland-Hayman inequality \cite{Friedland-Hayman_1976} to the $RCD(N-2,N-1)$ metric measure spaces $(\Sigma,d_{\Sigma},m_{\Sigma})$.
    \end{itemize}
    As an application of Theorem \ref{main_theorem}, we will prove the following rigidity result:
    \begin{theorem}\label{rigidity_theorem}
    	Let $(C(\Sigma),d_{\Sigma},m_{\Sigma})$ be as in Theorem \ref{main_theorem}. If there exist $u_{1},u_{2}$ as in Theorem \ref{main_theorem} such that, for some $0<r_{1}<r_{2}<\infty$,
    	\begin{equation}
    		0<J(r_{1})=J(r_{2}) < \infty,
    	\end{equation}
    	then $(\Sigma,d_{\Sigma},m_{\Sigma})$ is a spherical suspension, i.e. there exists an $RCD(N-3,N-2)$-space $(Y,d_{Y},m_{Y})$ such that $(\Sigma,d_{\Sigma},m_{\Sigma})$ is isometric to $[0,\pi] \times_{\sin}^{N-2} Y$. If, in addition, $(\Sigma,d_{\Sigma},m_{\Sigma})$ is an $(N-1)$-dimensional smooth Riemannian manifold, then we have that:
    	\begin{itemize}
    		\item[(1)] $(\Sigma,d_{\Sigma},m_{\Sigma})$ is isometric to the sphere $\partial B_{1} \subset \mathbb{R}^{N}$;
    		\item[(2)] there exist $k_{1},k_{2}>0$ and $\nu \in \partial B_{1}$ such that
    		\begin{equation}
    			u_{1}(x)=k_{1}(x \cdot \nu)^{+},\quad u_{2}(x)=k_{2}(x\cdot\nu)^{-},\quad \text{in } B_{r_{2}}(p).
    		\end{equation}
    	\end{itemize}
    \end{theorem}
    \begin{remark}
    	Recently, Chan, Zhang and Zhu extend the Weiss-type monotonicity formula to metric measure cones \cite{Chan-Zhang-Zhu_2021}. Moreover, they use this monotonicity formula to prove that the free boundary of one phase Bernoulli type free boundary problem on non-collapsed $RCD$ metric measure spaces is $C^{\alpha}$ regular away from a set of codimension $\ge 3$.
    \end{remark}

\section{Preliminaries}
    Throughout the paper, a metric measure space is a triple $(X,d,m)$, where $(X,d)$ is a complete, separable metric space, and $m$ is a positive Radon measure. We always assume that $\mathrm{supp} (m) = X$.
    \subsection{Calculus Tools on $RCD(K,N)$-spaces}
    Given $K \in \mathbb{R}, N \in [1,\infty]$, $CD(K,N)$-spaces, proposed by Sturm\cite{Sturm_2006_1,Sturm_2006_2} and Lott-Villani\cite{Lott-Villani_2009} independently, is a class of metric measure spaces satisfying the so-called curvature-dimension condition $CD(K,N)$, which is a generalized notion of Ricci curvature bounded from below by $K$ and dimension bounded from above by $N$. Since $CD(K,N)$-spaces include Finsler manifold, Ambrosio, Gigli, and Savar\'{e}\cite{Ambrosio-Gigli-Savare_2014} proposed the notion of $RCD(K,\infty)$-spaces later by adding the infinitesimally Hilbertian assumption to the $CD(K,\infty)$ condition, called Riemannian curvature-dimension condition, in order to rule out Finsler geometries. The finite dimensional counterpart $RCD(K,N)$-spaces was proposed by Gigli\cite{Gigli_2015}. Erbar-Kuwada-Sturm\cite{Erbar-Kuwada-Sturm_2015} and Ambrosio-Mondino-Savar\'{e}\cite{Ambrosio-Mondino-Savare_2016} introduce the notion of reduced Riemannian curvature-dimension condition $RCD^{*}(K,N)$, slitly weaker than $RCD(K,N)$ condition, and proved that $RCD^{*}(K,N)$ is equivalent to a weak formulation of Bocher inequality. Finally, Cavalletti and Milman\cite{Cavalletti-Milman_2021} show that $RCD(K,N)$ is equivalent to $RCD^{*}(K,N)$. We refer the readers to the survey\cite{Ambrosio_2018} and references therein for background of $RCD(K,N)$-spaces.
    
    We recall the notion of \emph{non-collapsed} $RCD(K,N)$-spaces introduced by De Philippis and Gigli \cite{DePhilippis-Gigli_2018}:
    \begin{definition}
    	We say that an $RCD(K,N)$-space $(X,d,m)$ is non-collapsed if $m = \mathcal{H}^{N}$, the $N$-dimensional Hausdorff measure. We also say that $(X,d,m)$ is a $ncRCD(K,N)$-space for short.
    \end{definition}
     If $(X,d,m)$ is an $ncRCD(K,N)$-space, then $N$ must be an integer, and in this case, it holds that
     \begin{equation}
     	m(B_{r}(x)) \le C_{K,N}r^{N}, \quad \forall x \in X, \forall r \in (0,1)
     \end{equation} 
     for some constant $C_{K,N} > 0$, depending on $K,N$, see \cite[Corollary 2.14]{DePhilippis-Gigli_2018}.
     
     \begin{definition}
     	Given a metric measure space $(\Sigma,d_{\Sigma},m_{\Sigma})$ with $\mathrm{diam}\Sigma \le 2\pi$ and $N \ge 2$, the metric measure cone over $(\Sigma,d_{\Sigma},m_{\Sigma})$ with vertex $p$ is defined to be the metric measure space $(C(\Sigma),d_{C},m_{C})$, where
     	\begin{itemize}
     		\item $C(\Sigma) := \Sigma \times (0,\infty) \cup \{p\}$;
     		\item For $(r_{1},\xi_{1}),(r_{2},\xi_{2}) \in C(\Sigma)$, where $r_{i} \in \mathbb{R}^{+}, \xi_{i} \in \Sigma$, the distance between them is defined to be
     		\begin{equation}
     			d_{C}\big((r_{1},\xi_{1}),(r_{2},\xi_{2})\big) := \sqrt{r_{1}^{2} + r_{2}^{2} - 2r_{1}r_{2}\cos d_{\Sigma}(\xi_{1},\xi_{2})}
     		\end{equation}
     	    \item $m_{C} := r^{N-1}\mathrm{d}r \otimes m_{\Sigma}$.
     	\end{itemize}
     \end{definition}
     The following theorem is taken from \cite[Corollary 1.3]{Ketterer_2015}.
     \begin{theorem}\label{theorem_1_preliminaries}
     	Let $(C(\Sigma),d_{C},m_{C})$ be the metric measure cone over the metric measure space $(\Sigma,d_{\Sigma},m_{\Sigma})$, and $N \ge 2$. Then $(C(\Sigma),d_{C},m_{C})$ is an $RCD(0,N)$-space if and only if $(\Sigma,d_{\Sigma},m_{\Sigma})$ is an $RCD(N-2,N-1)$-space and $\mathrm{diam}(\Sigma) \le \pi$.
     \end{theorem}
 
     Next, we introduce the notion of Sobolev spaces on $RCD$-spaces. We adopt the approach of Cheeger\cite{Cheeger_1999}. We point out that there are several different approach to define Sobolev spaces on metric measure spaces, see \cite{Cheeger_1999,Ambrosio-Gigli-Savare_2014_2,Shanmugalingam_2000,Hajlasz_Koskela_2000}, and all of them are equivalent to each other in $RCD$-spaces.
     
     Let $(X,d,m)$ be an $RCD(K,N)$-space. Given a Lipschitz function $f : X \to \mathbb{R}$, denoted by $f \in \mathrm{LIP}(X)$, we define the pointwise Lipschitz constant as 
     \begin{equation}
     	\mathrm{lip}(f)(x) := \begin{cases}
     		\limsup_{y \to x} \frac{\lvert f(x) - f(y)\rvert}{d(x,y)},\quad & \text{if } x \text{ is not isolated}\\
     		0,\quad & \text{if } x \text{ is isolated}
     	\end{cases}
     \end{equation}
     Given $f \in L^{2}(X,m)$, we define the Cheeger energy as
     \begin{equation}
     	\mathrm{Ch}(f) := \inf \{ \liminf_{i \to \infty} \frac{1}{2} \int_{X} \mathrm{lip}^{2}(f_{i}) \mathrm{d}m : \mathrm{LIP}(X) \ni f_{i} \to f \text{ in } L^{2}(X,m)\}.
     \end{equation}
     The Sobolev space $W^{1,2}(X)$ is defined as
     \begin{equation}
     	W^{1,2}(X) := \{f \in L^{2}(X,m) : \mathrm{Ch}(f) < \infty\},
     \end{equation}
     with the Sobolev norm 
     \begin{equation}
     	\Vert f \Vert_{W^{1,2}(X)} := \sqrt{\Vert f \Vert^{2}_{L^{2}(X,m)} + 2\mathrm{Ch}(f)}.
     \end{equation}
     The Sobolev space $(W^{1,2}(X),\Vert \cdot \Vert_{W^{1,2}(X)})$ is a Hilbert space whenever $(X,d,m)$ is an $RCD$-space, see \cite[Proposition 4.10, Theorem 5.1]{Ambrosio-Gigli-Savare_2014}.
     
     For each $f \in W^{1,2}(X)$, there exists a unique $L^{2}(X,m)$-function, denoted by $\lvert \nabla f \rvert$ and called the weak upper gradient of $f$, such that the Cheeger energy can be represented as 
     \begin{equation}
     	\mathrm{Ch}(f) = \frac{1}{2} \int_{X} \left| \nabla f \right| \mathrm{d}m,
     \end{equation}
     see \cite[Section 2]{Cheeger_1999}.
     
     Given $f,g \in W^{1,2}(X)$, it can be shown that the limit
     \begin{equation}
     	\langle \nabla f, \nabla g \rangle := \lim_{\varepsilon \to 0^{+}} \frac{1}{2\varepsilon} \big( \lvert \nabla (f + \varepsilon g) \rvert|^{2} - \lvert \nabla f \rvert^{2} \big)
     \end{equation}
     exists in $L^{1}(X,m)$, see \cite[Lemma 4.7]{Ambrosio-Gigli-Savare_2014}. Since $(X,d,m)$ is an $RCD$-space, the map
     \begin{equation}
     	W^{1,2}(X) \times W^{1,2}(X) \ni (f,g) \mapsto \langle \nabla f, \nabla g \rangle \in L^{1}(X,m)
     \end{equation}
     is symmetric, linear, and
     \begin{equation}
     	\left| \langle \nabla f, \nabla g \rangle \right| \le \left| \nabla f \right| \left| \nabla g \right|
     \end{equation}
     see \cite[Section 4.3]{Ambrosio-Gigli-Savare_2014},\cite[Section 4.3]{Gigli_2015}. Moreover, the Leibniz rule holds:
     \begin{equation}
     	\int_{X} \langle \nabla f, \nabla (gh) \rangle \mathrm{d}m = \int_{X} h \langle \nabla f, \nabla g \rangle \mathrm{d}m + \int_{X} g \langle \nabla f, \nabla h \rangle \mathrm{d}m
     \end{equation}
     for all $f,g,h \in W^{1,2}(X) \cap L^{\infty}(X,m)$.
     
     We denote by $\mathrm{LIP}_{0}(X)$ the collection of all Lipschitz function with compact support. We say that a function $f$ is locally $W^{1,2}(X)$, denoted by $f \in W^{1,2}_{\mathrm{loc}}(X)$, if $\phi f \in W^{1,2}(X)$ for all $\phi \in \mathrm{LIP}_{0}(X)$.
     
     We recall the notion of distributional Laplacian introduced by Gigli\cite{Gigli_2015}:
     \begin{definition}\label{Def_distributional_Laplacian}
     	Given $f \in W^{1,2}_{\mathrm{loc}}(X)$, its distributional Laplacian $\mathbf{\Delta} f$ is a linear functional on $\mathrm{LIP}_{0}(X)$ defined as
     	\begin{equation}
     		\mathbf{\Delta} f (\phi) := - \int_{X} \langle \nabla f, \nabla \phi \rangle \mathrm{d}m
     	\end{equation}
        for all $\phi \in \mathrm{LIP}_{0}(X)$. If there exists a signed Radon measure $\mu$ such that
        \begin{equation}
        	\mathbf{\Delta} f (\phi) = \int_{X} \phi \mathrm{d}\mu
        \end{equation}
        for all $\phi \in \mathrm{LIP}_{0}(X)$, then we say that $\mathbf{\Delta} f = \mu$ in the sense of distribution, denoted by $f \in \mathrm{D}(\mathbf{\Delta})$.
     \end{definition}
 
     The Laplacian $\mathbf{\Delta}$ is a linear operator whenever $(X,d,m)$ is an $RCD(K,N)$-space. And in this case, since $\mathrm{LIP}_{0}(X)$ is dense in $W^{1,2}(X)$, the signed Radon measure $\mu$ in Definition \ref{Def_distributional_Laplacian} is unique, and we write $\mathbf{\Delta} f := \mu$. The chain rule and the Leibniz rule also hold for the $\mathbf{\Delta}$, see \cite[Section 4.2]{Gigli_2015}. In general, $\mathbf{\Delta} f$ may not absolutely continuous with respect to $m$. We consider the Radon-Nikodym decomposition of $\mathbf{\Delta} f$
     \begin{equation}
     	\mathbf{\Delta} f = (\Delta f)m + \Delta^{s} f
     \end{equation} 
     where $\Delta f$ denotes the density of $\mathbf{\Delta} f$ with respect to $m$, $\Delta^{s} f$ denotes the singular part. When $\mathbf{\Delta} f \ll m$, we have $\mathbf{\Delta} f = (\Delta f)m$ for some $\Delta f \in L^{2}(X,m)$, and we write $f \in \mathrm{D}(\Delta)$.
     
     Given $f \in \mathrm{D}(\mathbf{\Delta})$, we say that $f$ is subharmonic, denoted by $\mathbf{\Delta} f \ge 0$, if 
     \begin{equation}
     	\mathbf{\Delta}f (\phi) = - \int_{X} \langle \nabla f, \nabla \phi \rangle \mathrm{d}m \ge 0
     \end{equation}
     for all $0 \le \phi \in \mathrm{LIP}_{0}(X)$. $f$ is superharmonic if $-f$ is subharmonic, and is harmonic if it is both subharmonic and superharmonic.
     
     Now, restricted to the $RCD(0,N)$ (with $N \ge 2$) metric measure cone $(C(\Sigma),d_{C},m_{C})$ over a metric measure space $(\Sigma,d_{\Sigma},m_{\Sigma})$, with vertex $p$. Recall that $(\Sigma,d_{\Sigma},m_{\Sigma})$ is an $RCD(N-2,N-1)$-space. We will write $C(\Sigma) \ni x = (r,\xi)$, where $r \in \mathbb{R}^{+}, \xi \in \Sigma$, and use the notation $\nabla_{\mathbb{R}^{+}}, \nabla_{\Sigma}$ to denote the gradient on $\mathbb{R}^{+},\Sigma$, respectively. Similiar for $\mathbf{\Delta}_{\Sigma}$. By \cite[Proposition 3.4]{Ketterer_2015}, it holds that
     \begin{equation}
     	\int_{C(\Sigma)} \lvert \nabla u \rvert^{2} \mathrm{d}m_{C} = \int_{\mathbb{R}^{+}} r^{N} \lvert \nabla_{\mathbb{R}^{+}} u(\cdot,\xi) \rvert^{2}(r) \mathrm{d}r + \int_{\Sigma} r^{N-2} \lvert \nabla_{\Sigma} u(r,\cdot) \rvert^{2}(\xi) \mathrm{d}m_{\Sigma}
     \end{equation}
     for all $u \in W^{1,2}(C(\Sigma))$. Therefore, for $m_{C}$-a.e. $(r,\xi) \in C(\Sigma)$, it holds that
     \begin{equation}
     	\lvert \nabla u \rvert^{2}(r,\xi) = \lvert \nabla_{\mathbb{R}^{+}} u(\cdot,\xi) \rvert^{2}(r) + r^{-2} \lvert \nabla_{\Sigma} u(r,\cdot) \rvert^{2}(\xi)
     \end{equation}
     In particular, by polarization, for all $u,v \in W^{1,2}(X)$, it holds that
     \begin{equation}
     	\begin{split}
     		\langle \nabla u, \nabla v \rangle (r,\xi) = & \frac{1}{4} \Big( \lvert \nabla (u + v) \rvert^{2}(r,\xi) - \lvert \nabla (u - v) \rvert^{2}(r,\xi) \Big)\\
     		= & \frac{1}{4} \Big( \lvert \nabla_{\mathbb{R}^{+}} (u + v)(\cdot,\xi) \rvert^{2}(r) - \lvert \nabla_{\mathbb{R}^{+}} (u - v)(\cdot,\xi) \rvert^{2}(r) \Big)\\
     		& + \frac{1}{4} r^{-2} \Big( \lvert \nabla_{\Sigma} (u + v)(r,\cdot) \rvert^{2}(\xi) + \lvert \nabla_{\Sigma} (u - v)(r,\cdot) \rvert^{2}(\xi) \Big)\\
     		= & \langle \nabla_{\mathbb{R}^{+}} u(\cdot,\xi), \nabla_{\mathbb{R}^{+}} v(\cdot,\xi) \rangle (r) + r^{-2} \langle \nabla_{\Sigma} u(r,\cdot), \nabla_{\Sigma} v(r,\cdot) \rangle (\xi)
     	\end{split}
     \end{equation}
     for $m_{C}$-a.e. $(r,\xi) \in C(\Sigma)$.
     
     We recall a result of Mondino-Somela \cite{Mondino-Semola_2020}, restricted to the $RCD(N-2,N-1)$-space $(\Sigma,d_{\Sigma},m_{\Sigma})$ ($N \ge 2$ is an integer): given an open domain $\Gamma \subset \Sigma$, denote by $\lambda(\Gamma)$ the first eigenvalue, that is
     \begin{equation}
     	\lambda(\Gamma) := \inf \frac{\int_{\Gamma} \vert \nabla f \rvert^{2} \mathrm{d}m_{\Sigma}}{\int_{\Gamma} \lvert f \rvert^{2} \mathrm{d}m_{\Sigma}}
     \end{equation}
     where the $\inf$ is taken over all $f \in \mathrm{LIP}_{0}(\Gamma)$ with $f \not\equiv 0$. Let $\overline{\Gamma}$ be a geodesic ball in the standard sphere $\partial B_{1} \subset \mathbb{R}^{N}$ with
     \begin{equation}
     	\frac{m_{\Sigma}(\Gamma)}{m_{\Sigma}(\Sigma)} = \frac{\mathrm{Vol}(\overline{\Gamma})}{\mathrm{Vol}(\partial B_{1})}
     \end{equation}
  and $\lambda(\overline{\Gamma})$ be the first eigenvalue of $\overline{\Gamma}$. Then we have the following theorem \cite[Theorem 1.5, Theorem 1.6, Corollary 1.7]{Mondino-Semola_2020}:
     \begin{theorem}\label{Faber_Krahn}
     	It holds that $\lambda(\Gamma) \ge \lambda(\overline{\Gamma})$. If the equality holds, then $(\Sigma,d_{\Sigma},m_{\Sigma
     	})$ is a spherical suspension, i.e. there exists an $RCD(N-3,N-2)$-space $(Y,d_{Y},m_{Y})$ such that $(\Sigma,d_{\Sigma},m_{\Sigma})$ is isomorphic to $[0,\pi]\times_{\sin}^{N-2}Y$. If, in addition, $(\Sigma,d_{\Sigma},m_{\Sigma})$ is an $(N-1)$-dimensional smooth Riemannian manifold, then $(\Sigma,d_{\Sigma},m_{\Sigma})$ is isometric to the sphere $\partial B_{1} \subset \mathbb{R}^{N}$.
     \end{theorem}
\subsection{Friedland-Hayman Inequality}
At first, we recall the notion of characteristic constant. Let $\Gamma \subset \partial B_{1}$ be open and let $C(\Gamma)$ be the cone generated by $\Gamma$ with vertex $0$, i.e.
\begin{equation}
	C(\Gamma) := \{r\theta: r >0,\theta \in \Gamma\}.
\end{equation}
Consider a homogeneous harmonic function $h$ on $C(\Gamma)$, i.e. $h$ is harmonic in $C(\Gamma)$ and has the form
\begin{equation}
	h(r\theta) = r^{\alpha}f(\theta),\quad \text{for some } \alpha>0.
\end{equation}
We also assume that $h=0$ on $\partial C(\Gamma)$. If we let $\Delta_{\theta}$ be the Laplace-Beltrami operator on $\partial B_{1}$, then we have that
\begin{equation}
	\Delta h = r^{\alpha-2}\Big(\big(\alpha(n-2+\alpha)f\big)+\Delta_{\theta}f \Big).
\end{equation}
Then $h$ is harmonic if and only if $f$ is an eigenfunction of $\Delta_{\theta}$, i.e.
\begin{equation}
	-\Delta_{\theta}f = \alpha(n-2+\alpha)f.
\end{equation}
If $h>0$, then $\alpha(n-2+\alpha)=\lambda(\Gamma)$ is the first eigenvalue. Then corresponding $\alpha$ is called the characteristic constant of $\Gamma$, denoted by $\alpha(\Gamma)$.

We will need the following Friedland-Haymen inequality. It has been first established by Friedland and Hayman \cite{Friedland-Hayman_1976} and later by Kawohl \cite{Kawohl_1986} for the case of equality when $n>3$. See also the book \cite[Chapter 12]{Caffarelli-Salsa_2005} for a proof.
\begin{theorem}\label{Friedland-Haymen_inequality}
	Let $\Gamma_{1},\Gamma_{2}$ be two disjoint open subsets of $\partial B_{1}$. Then it holds that
	\begin{equation}
		\alpha(\Gamma_{1})+\alpha(\Gamma_{2}) \ge 2.
	\end{equation}
	The equality holds if and only if $\Gamma_{1},\Gamma_{2}$ are half-spheres.
\end{theorem}

\section{ACF-Monotonicity Formula}
\subsection{Stokes Formula on Annuli}\label{subsection_Stokes_formula}
    In this section, we give two Stokes formulas on $RCD(K,N)$-spaces that will be needed later in the proof of Theorem \ref{main_theorem}. The first one is taken from \cite[Remark 5.2]{Chan-Zhang-Zhu_2021}:
    \begin{lemma}\label{Stokes_1}
    	Let $(X,d,m)$ be an $RCD(K,N)$-space, $K \in \mathbb{R}, N \in (1,\infty)$. Let $\overline{B_{R_{2}}(x_{0})} \setminus B_{R_{1}}(x_{0}) \subset X$, $\rho(x) := d(x,x_{0})$, $\phi \in C^{2}([R_{1},R_{2}])$, and let $\psi := \phi(\rho)$. Suppose that
    	\begin{equation}
    		u \in C(\overline{B_{R_{2}}(x_{0})} \setminus B_{R_{1}}(x_{0})) \cap W^{1,2}(B_{R_{2}}(x_{0}) \setminus B_{R_{1}}(x_{0})),
    	\end{equation}
        and suppose that $\mathbf{\Delta} \psi$ is a signed Radon measure. Then it holds that
        \begin{equation}
        	\begin{split}
        		\int_{B_{r_{2}}(x_{0}) \setminus B_{r_{1}}(x_{0})} u \mathrm{d} \mathbf{\Delta} \psi = & - \int_{B_{r_{2}}(x_{0}) \setminus B_{r_{1}}(x_{0})} \langle \nabla u, \nabla \psi \rangle \mathrm{d}m\\
        		& + {\phi}'(r_{2}) \left. \frac{\mathrm{d}}{\mathrm{d}s} \right|_{s = r_{2}} \int_{B_{s}(x_{0})} u \mathrm{d}m - {\phi}'(r_{1}) \left. \frac{\mathrm{d}}{\mathrm{d}s} \right|_{s = r_{1}} \int_{B_{s}(x_{0})} u \mathrm{d}m
        	\end{split}
        \end{equation}
        for almost all $R_{1} < r_{1} < r_{2} < R_{2}$.
    \end{lemma}
    The second one is the following lemma:
    \begin{lemma}\label{Stokes_2}
    	Let $(X,d,m)$ be an $RCD(K,N)$-space, $K \in \mathbb{R}, N \in (1,\infty)$. Let $\overline{B_{R_{2}}(x_{0})} \setminus B_{R_{1}}(x_{0}) \subset X$, $\rho(x) := d(x,x_{0})$, $\phi \in C^{1}([R_{1},R_{2}])$, and let $\psi := \phi(\rho)$. Suppose that
    	\begin{equation}
    		u \in C(\overline{B_{R_{2}}(x_{0})} \setminus B_{R_{1}}(x_{0})) \cap W^{1,2}(B_{R_{2}}(x_{0}) \setminus B_{R_{1}}(x_{0})),
    	\end{equation}
    	and suppose that $\mathbf{\Delta} u$ is a signed Radon measure. Then it holds that
    	\begin{equation}
    		\begin{split}
    			\int_{B_{r_{2}}(x_{0}) \setminus B_{r_{1}}(x_{0})} \psi \mathrm{d} \mathbf{\Delta} u = & - \int_{B_{r_{2}}(x_{0}) \setminus B_{r_{1}}(x_{0})} \langle \nabla u , \nabla \psi \rangle \mathrm{d}m\\
    			& + \phi(r_{2}) \left. \frac{\mathrm{d}}{\mathrm{d}s} \right|_{s = r_{2}} \int_{B_{s}(x_{0})} \langle \nabla u, \nabla \rho \rangle \mathrm{d}m - \phi(r_{1}) \left. \frac{\mathrm{d}}{\mathrm{d}s} \right|_{s = r_{1}} \int_{B_{s}(x_{0})} \langle \nabla u, \nabla \rho \rangle \mathrm{d}m
    		\end{split}
    	\end{equation}
        for almost all $R_{1} < r_{1} < r_{2} < R_{2}$.
    \end{lemma}
    \begin{proof}
    	Since both $\mathbf{\Delta} u$ and $m$ are Radon measure, we have
    	\begin{equation}
    		\begin{cases}
    			\lim\limits_{j \to \infty} \left| \mathbf{\Delta} u \right|\big(\overline{B_{r + j^{-1}}(x_{0})} \setminus B_{r}(x_{0})\big) = 0,\\
    			\lim\limits_{j \to \infty} m\big(\overline{B_{r + j^{-1}}(x_{0})} \setminus B_{r}(x_{0})\big) = 0,
    		\end{cases}
    	\end{equation}
        for almost all $r \in (R_{1},R_{2})$. On the other hand, we have that
        \begin{equation}
        	\left| \int_{B_{r}(x_{0}) \setminus B_{s}(x_{0})} \langle \nabla u, \nabla \rho \rangle \mathrm{d}m \right| \le \sqrt{m(B_{R_{2}}(x_{0}))} \Big( \int_{B_{r}(x_{0}) \setminus B_{s}(x_{0})} \left| \nabla u \right|^{2} \mathrm{d}m \Big)^{1/2}
        \end{equation}
        Since $u \in W^{1,2}(B_{R_{2}}(x_{0})\setminus B_{R_{1}}(x_{0}))$, we get that the function
        \begin{equation}
        	r \mapsto \int_{B_{r}(x_{0})} \langle \nabla u, \nabla \rho \rangle \mathrm{d}m
        \end{equation}
    is locally absolutely continuous and hence differentiable for almost all $r \in (R_{1},R_{2})$.
    
    Now we fix such $r_{1} < r_{2}$, and, for $j \in \mathbb{N}$, define
    \begin{equation}
    	\eta_{j} (t) := \begin{cases}
    		0, & \text{if } t \in [0,r_{1} - j^{-1}] \cup [r_{2} + j^{-1},R_{2}],\\
    		j(t - r_{1} + j^{-1}), & \text{if } t \in (r_{1} - j^{-1},r_{1}),\\
    		1, & \text{if } t \in [r_{1},r_{2}],\\
    		1 - j(t -r_{2}), & \text{if } t \in (r_{2},r_{2} + j^{-1}).
    	\end{cases}
    \end{equation}
    And, we let $\psi_{j} := \eta_{j}(\rho) \psi \in W^{1,2}_{0}(B_{R_{2}}(x_{0}) \setminus B_{R_{1}}(x_{0}))$.
    
    Then, on the one hand, we have that
    \begin{equation}
    	\begin{split}
    		\int_{B_{R_{2}}(x_{0})} \psi_{j} \mathrm{d}\mathbf{\Delta}u & = \int_{B_{r_{2}}(x_{0}) \setminus B_{r_{1}}(x_{0})} \psi \mathrm{d}\mathbf{\Delta}u\\
    		& + \Big(\int_{B_{R_{2}}(x_{0}) \setminus B_{r_{2}}(x_{0})} + \int_{B_{r_{1}}(x_{0})} \Big) \eta_{j}(\rho) \psi \mathrm{d}\mathbf{\Delta}u.
    	\end{split}
    \end{equation}
    Note that
    \begin{equation}
    	\begin{split}
    		\left| \int_{B_{R_{2}}(x_{0}) \setminus B_{r_{2}}(x_{0})} \eta_{j}(\rho) \psi \mathrm{d}\mathbf{\Delta}u \right| & \le \int_{B_{r_{2} + j^{-1}}(x_{0}) \setminus B_{r_{2}}(x_{0})} \Big(\sup_{[R_{1},R_{2}]} \lvert \phi \rvert \Big) \mathrm{d}\mathbf{\Delta}u\\
    		& \le \Big(\sup_{[R_{1},R_{2}]}  \lvert \phi \rvert \Big) \lvert \mathbf{\Delta} u \rvert(\overline{B_{r_{2} + j^{-1}}(x_{0})} \setminus B_{r_{2}}(x_{0}))\\
    		& \to 0,\quad \text{as } j \to \infty.
    	\end{split}
    \end{equation}
    Similiarly, $\lim\limits_{j \to \infty} \int_{B_{r_{1}}(x_{0})} \eta_{j}(\rho) \psi \mathrm{d}\mathbf{\Delta}u = 0$.
    
    Therefore, we get that
    \begin{equation}\label{step_1}
    	\lim_{j \to \infty} \int_{B_{R_{2}}(x_{0})} \psi_{j} \mathrm{d}\mathbf{\Delta}u = \int_{B_{r_{2}}(x_{0}) \setminus B_{r_{1}}(x_{0})} \psi \mathrm{d}\mathbf{\Delta}u.
    \end{equation}
    On the other hand, we have that
    \begin{equation}
    	\begin{split}
    		\int_{B_{R_{2}}(x_{0})} \psi_{j} \mathrm{d}\mathbf{\Delta}u & = - \int_{B_{R_{2}}(x_{0})} \langle \nabla u, \nabla \psi_{j} \rangle \mathrm{d}m\\
    		& = - \int_{B_{R_{2}}(x_{0})} \langle \nabla u, \eta_{j} \nabla \psi + \psi \nabla \eta_{j} \rangle \mathrm{d}m\\
    		& = - \int_{B_{r_{2}}(x_{0}) \setminus B_{r_{1}}(x_{0})} \langle \nabla u, \nabla \psi \rangle \mathrm{d}m\\
    		& - \Big(\int_{B_{R_{2}(x_{0})} \setminus B_{r_{2}}(x_{0})} + \int_{B_{r_{1}}(x_{0})} \Big) \langle \nabla u, \nabla \psi \rangle \eta_{j} \mathrm{d}m\\
    		& - \int_{B_{R_{2}}(x_{0})} \langle \nabla u, \nabla \eta_{j} \rangle \psi \mathrm{d}m
    	\end{split}
    \end{equation}
    We estimate $\int_{B_{R_{2}(x_{0})} \setminus B_{r_{2}}(x_{0})} \langle \nabla u, \nabla \psi \rangle \eta_{j} \mathrm{d}m$ as follow:
    \begin{equation}
    	\begin{split}
    		\left| \int_{B_{R_{2}(x_{0})} \setminus B_{r_{2}}(x_{0})} \langle \nabla u, \nabla \psi \rangle \eta_{j} \mathrm{d}m \right| & \le \int_{B_{r_{2} + j^{-1}}(x_{0}) \setminus B_{r_{2}}(x_{0})} \lvert \nabla u \rvert \lvert \nabla \psi \rvert \mathrm{d}m\\
    		& \le \Big(\sup_{[R_{1},R_{2}]} \lvert {\phi}' \rvert \Big) \lvert \nabla u \rvert_{L^{2}(B_{r_{2} + j^{-1}}(x_{0}) \setminus B_{r_{2}}(x_{0}))} \Big(m(B_{r^{2}+j^{-1}}(x_{0})\setminus B_{r_{2}}(x_{0}))\Big)^{1/2}\\
    		& \to 0,\quad \text{as } j \to \infty.
    	\end{split}
    \end{equation}
    Similiarly, we have that $\lim\limits_{j \to \infty} \int_{B_{r_{1}}(x_{0})} \langle \nabla u, \nabla \psi \rangle \eta_{j} \mathrm{d}m = 0$.
    
    While for $\int_{B_{R_{2}}(x_{0})} \langle \nabla u, \nabla \eta_{j} \rangle \psi \mathrm{d}m$, we have that
    \begin{equation}
    	\begin{split}
    		\int_{B_{R_{2}}(x_{0})} \langle \nabla u, \nabla \eta_{j} \rangle \psi \mathrm{d}m & = \Big(- \int_{B_{r_{2} + j^{-1}}(x_{0}) \setminus B_{r_{2}}(x_{0})} + \int_{B_{r_{1}}(x_{0}) \setminus B_{r_{1} - j^{-1}}(x_{0})} \Big) \langle \nabla u, j \nabla \rho \rangle \psi \mathrm{d}m\\
    		& = - \int_{B_{r_{2} + j^{-1}}(x_{0}) \setminus B_{r_{2}}(x_{0})} j \langle \nabla u, \nabla \rho \rangle (\phi(r_{2}) + \phi - \phi(r_{2})) \mathrm{d}m\\
    		& + \int_{B_{r_{1}}(x_{0}) \setminus B_{r_{1} - j^{-1}}(x_{0})} j \langle \nabla u, \nabla \rho \rangle (\phi(r_{1}) + \phi - \phi(r_{1})) \mathrm{d}m
    	\end{split}
    \end{equation}
    Since
    \begin{equation}
    	\begin{split}
    		\left| \int_{B_{r_{2} + j^{-1}}(x_{0}) \setminus B_{r_{2}}(x_{0})} j \langle \nabla u, \nabla \rho \rangle (\phi - \phi(r_{2})) \mathrm{d}m \right| & \le \Big(\sup_{[R_{1},R_{2}]} \lvert {\phi}' \rvert\Big) \int_{B_{r_{2} + j^{-1}}(x_{0}) \setminus B_{r_{2}}(x_{0})} \lvert \nabla u \rvert \lvert \nabla \rho \rvert \mathrm{d}m\\
    		& \to 0,\quad \text{as } j \to \infty.
    	\end{split}
    \end{equation}
    Similiarly, $\lim_{j \to \infty} \int_{B_{r_{1}}(x_{0}) \setminus B_{r_{1} - j^{-1}}(x_{0})} j \langle \nabla u, \nabla \rho \rangle \big(\phi - \phi(r_{1})\big) \mathrm{d}m = 0$.
    
    Therefore, we get that 
    \begin{equation}\label{step_2}
    	\begin{split}
    		\lim_{j \to \infty} \int_{B_{R_{2}}(x_{0})} \psi_{j} \mathrm{d}\mathbf{\Delta}u & = - \int_{B_{r_{2}}(x_{0}) \setminus B_{r_{1}}(x_{0})} \langle \nabla u, \nabla \psi \rangle \mathrm{d}m\\
    		& + \phi(r_{2}) \left. \frac{\mathrm{d}}{\mathrm{d}s} \right|_{s = r_{2}} \int_{B_{s}(x_{0})} \langle \nabla u, \nabla \rho \rangle \mathrm{d}m - \phi(r_{1}) \left. \frac{\mathrm{d}}{\mathrm{d}s} \right|_{s = r_{1}} \int_{B_{s}(x_{0})} \langle \nabla u, \nabla \rho \rangle \mathrm{d}m
    	\end{split}
    \end{equation}
    The conclusion follows from \eqref{step_1}\eqref{step_2}.
    \end{proof}

\subsection{Proof of Theorem \ref{main_theorem}}\label{proof_of_main}
    At first, we recall the assumption of Theorem \ref{main_theorem} and fix some notation. Let $N \ge 2$ be an integer and let $(C(\Sigma),d_{C},m_{C})$ be the metric measure cone over an $RCD(N-2,N-1)$ metric measure space $(\Sigma,d_{\Sigma},m_{\Sigma})$ with vertex $p$. Recall that $(C(\Sigma),d_{\Sigma},m_{\Sigma})$ is an $RCD(0,N)$ metric measure space (Theorem \ref{theorem_1_preliminaries}). Let $u_{1},u_{2} : C(\Sigma) \to [0,\infty)$ be continuous, non-negative and $u_{1},u_{2} \in W^{1,2}_{\mathrm{loc}}(C(\Sigma))$. We assume that $u_{1},u_{2}$ satisfy the following conditions:
    \begin{itemize}
    	\item[(1)] $u_{1}(p)=u_{2}(p)=0$;
    	\item[(2)] $u_{1}u_{2}=0$ in $C(\Sigma)$;
    	\item[(3)] $\mathbf{\Delta} u_{i} \ge 0$ in $C(\Sigma)$, $i=1,2$.
    \end{itemize}
    For $r \in (0,\infty)$, we define the quantities
    \begin{equation}
     	\begin{split}
     		& A_{i}(r) := \int_{B_{r}(p)} \frac{\lvert \nabla u_{i} \rvert^{2}}{d(x,p)^{N-2}} \mathrm{d}m_{C},\quad i=1,2,\\
     		& J(r) := \frac{1}{r^{4}} A_{1}(r)A_{2}(r).
     	\end{split}
    \end{equation}
    Note that if $A_{i}(r_{0})=\infty$ for some $r_{0} \in (0,\infty)$, then $A_{i}(r)=\infty$ for all $r \in (r_{0},\infty)$. In this case, $J(r)\equiv \infty$ is trivially monotone increasing in $r \in [r_{0},\infty)$. Therefore, we always assume that $A_{i}(r) < \infty$. Then, by the coarea formular \cite[Remark 4.3]{Miranda_2003}, we have that
    \begin{equation}
        \begin{split}
     		A_{i}(r) & = \int_{0}^{r} \int_{C(\Sigma)} \frac{\lvert \nabla u_{i} \rvert^{2}(\rho,\xi)}{\rho^{N-2}} \mathrm{d}\lvert D \chi_{B_{\rho}(p)} \rvert \mathrm{d}\rho\\
     		& = \int_{0}^{r}\rho^{N-1}\int_{\Sigma} \frac{\lvert \nabla u_{i} \rvert^{2}(\rho,\xi)}{\rho^{N-2}} \mathrm{d}m_{\Sigma}(\xi) \mathrm{d}\rho\\
     		& = \int_{0}^{r} \int_{\Sigma} \rho \lvert \nabla u_{i} \rvert^{2}(\rho,\xi) \mathrm{d}m_{\Sigma}(\xi) \mathrm{d}\rho.
        \end{split}
    \end{equation}
    Therefore, we get that following lemma:
    \begin{lemma}\label{lemma_1_proof_of_main_theorem}
        If $A_{i}(r_{0}) < \infty$ for some $r_{0} >0$, then $A_{i}(r)$ is absolutely continuous in $r \in (0,r_{0})$ and thus differentiable for almost all $r \in (0,r_{0})$. Moreover, at a differentiable point $r \in (0,r_{0})$, it holds that
     	\begin{equation}
     		A_{i}^{\prime}(r) = \int_{\Sigma} r \lvert \nabla u_{i} \rvert^{2}(r,\cdot) \mathrm{d}m_{\Sigma}.
     	\end{equation}
    \end{lemma}
    A direct corollary is that:
    \begin{corollary}\label{corollary_1_proof_of_main_theorem}
     	Suppose that $A_{i}(r_{0}) < \infty$ for some $r_{0} \in (0,\infty)$, $i=1,2$. Then $J(r)$ is absolutely continuous in $r \in (0,r_{0})$ and thus differentiable for almost all $r \in (0,r_{0})$. Moreover, at a differentiable point $r \in (0,r_{0})$, it holds that
     	\begin{equation}
     		\frac{J^{\prime}(r)}{J(r)} = \frac{A_{1}^{\prime}(r)}{A_{1}(r)} + \frac{A_{2}^{\prime}(r)}{A_{2}(r)} - \frac{4}{r}.
     	\end{equation}
    \end{corollary}
    Now, we define
    \begin{equation}
    	\Gamma_{i}(r) := \{\xi \in \Sigma : u_{i}(r,\xi) > 0\},\quad \forall r \in (0,\infty),i=1,2.
    \end{equation}
    Then, since $u_{1},u_{2}$ are continuous and $u_{1}u_{2}=0$ in $C(\Sigma)$, we get that $\Gamma_{1}(r),\Gamma_{2}(r)$ are disjoint open subsets of $\Sigma$ for all $r \in (0,\infty)$. The next step is to estimate $A_{i}(r)$ and $A_{i}^{\prime}(r)$.
    \begin{lemma}\label{lemma_2_proof_of_main_theorem}
    	For almost all $r \in (0,\infty)$, it holds that
    	\begin{equation}
    		A_{i}^{\prime}(r) \ge \int_{\Gamma_{i}(r)} r\lvert \nabla_{\mathbb{R}^{+}} u_{i}(\cdot,\xi) \rvert^{2} (r) + \lambda(\Gamma_{i}(r))r^{-1} \lvert u_{i}(r,\xi) \rvert^{2} \mathrm{d}m_{\Sigma}.
    	\end{equation}
    	Recall that $\lambda(\Gamma_{i}(r))$ is the first eigenvalue of $\Gamma_{i}(r)$.
    \end{lemma}
    \begin{proof}
    	Recall that
    	\begin{equation}
    		\lvert \nabla u_{i} \rvert^{2}(r,\xi) = \lvert \nabla_{\mathbb{R}^{+}} u_{i} (\cdot,\xi)\rvert^{2} + r^{-2} \lvert \nabla_{\Sigma} u_{i}(r,\cdot) \rvert^{2}(\xi).
    	\end{equation}
    	The conclusion follows from the definition of $\lambda(\Gamma_{i}(r))$.
    \end{proof}
    \begin{lemma}\label{lemma_3_proof_of_main_theorem}
    	$\mathbf{\Delta} d(x,p)^{2-N} \equiv 0$ in $C(\Sigma) \setminus \{p\}$ as a Radon measure.
    \end{lemma}
    \begin{proof}
    	Let $f(r,\theta) = f(x):=d(x,p)$. For each $\phi \in \mathrm{LIP}_{0}(C(\Sigma)\setminus\{p\})$, it holds that
    	\begin{equation}
    		\begin{split}
    			&\int_{C(\Sigma)} \langle \nabla \phi ,\nabla f^{2-N} \rangle \mathrm{d}m_{C}\\
    			= & \int_{C(\Sigma)} (2-N)f^{1-N} \langle \nabla \phi ,\nabla f \rangle \mathrm{d}m_{C}\\
    			= & \int_{0}^{\infty} \int_{\Sigma} (2-N)r^{1-N} \Big( \langle \nabla_{\mathbb{R}^{+}} \phi, \nabla_{\mathbb{R}^{+}} f \rangle + r^{-2}\langle \nabla_{\Sigma} \phi, \nabla_{\Sigma} f \rangle \Big) r^{N-1} \mathrm{d}m_{\Sigma}\mathrm{d}r.
    		\end{split}
    	\end{equation}
    	Note that $f (r,\cdot) \equiv r$ on $\{r\} \times \Sigma$. Thus $\nabla_{\Sigma} f \equiv 0, \nabla_{\mathbb{R}^{+}} f \equiv 1$ and therefore
    	\begin{equation}
    		\begin{split}
    			\int_{C(\Sigma)} \langle \nabla \phi, \nabla f^{2-N}\rangle \mathrm{d}m_{C} =& (2-N) \int_{0}^{\infty} \int_{\Sigma} \frac{\mathrm{d}}{\mathrm{d}r} \phi \mathrm{d}m_{\Sigma}\mathrm{d}r\\
    			= & (2-N)\int_{\Sigma} \Big(\lim_{r \to \infty} \phi (r,\theta) - \lim_{r\to 0} \phi(r,\theta) \Big) \mathrm{d}m_{\Sigma}\\
    			= & 0,
    		\end{split}
    	\end{equation}
    	i.e. $\mathbf{\Delta} d(x,p)^{2-N} = 0$ in $C(\Sigma) \setminus \{p\}$ as a Radon measure.
    \end{proof}
    \begin{lemma}\label{lemma_4_proof_of_main_theorem}
    	For almost all $r \in (0,\infty)$, it holds that
    	\begin{equation} 
    		A_{i}(r) \le r \int_{\Gamma_{i}(r)} u_{i}(r,\xi) \lvert \nabla_{\mathbb{R}^{+}} u_{i}(\cdot,\xi) \rvert(r) + \frac{N-2}{2}r^{-1}u_{i}^{2}(r,\xi) \mathrm{d}m_{\Sigma}.
    	\end{equation}
    \end{lemma}
    \begin{proof}
    	Since $u_{i}\mathbf{\Delta} u_{i} \ge 0$ in $C(\Sigma)$, we get that
    	\begin{equation}
    		\mathbf{\Delta} \frac{u_{i}^{2}}{2} = \frac{1}{2}u_{i}\mathbf{\Delta} u_{i} + \lvert \nabla u_{i} \rvert^{2} \ge \lvert u_{i} \rvert^{2},\quad \text{in } C(\Sigma).
    	\end{equation}
    	Therefore, we get that
    	\begin{equation}
    		A_{i}(r) \le \int_{B_{r}(p)} \frac{1}{d(x,p)^{N-2}} \mathrm{d}\mathbf{\Delta} \frac{u_{i}^{2}}{2}.
    	\end{equation}
    	By Lemma \ref{Stokes_1} and Lemma \ref{Stokes_2}, we have that, for almost all $0<r_{1}<r_{2}<r$,
    	\begin{equation}\label{equation_1_lemma_4_proof_of_main_theorem}
    		\begin{split}
    			& \int_{B_{r_{2}}(p)\setminus B_{r_{1}}(p)} \frac{u_{i}^{2}}{2} \mathrm{d} \mathbf{\Delta} d(x,p)^{2-N} - \frac{2-N}{2} \int_{\Sigma} u_{i}^{2}(r_{2},\cdot) \mathrm{d}m_{\Sigma} + \frac{2-N}{2} \int_{\Sigma} u_{i}^{2}(r_{1},\cdot) \mathrm{d}m_{\Sigma}\\
    			= & - \int_{B_{r_{2}}(p)\setminus B_{r_{1}}(p)} \langle \nabla \frac{u_{i}^{2}}{2},\nabla d(x,p)^{2-N} \rangle \mathrm{d}m_{C}\\
    			= & \int_{B_{r_{2}}(p) \setminus B_{r_{1}}(p)} d(x,p)^{2-N} \mathrm{d}\mathbf{\Delta} \frac{u_{i}^{2}}{2} -r_{2} \int_{\Sigma} u_{i}(r_{2},\cdot) \nabla_{\mathbb{R}^{+}} u_{i}(r_{2},\cdot) \mathrm{d}m_{\Sigma} + r_{1} \int_{\Sigma} u_{i}(r_{1},\cdot) \nabla_{\mathbb{R}^{+}} u_{i}(r_{1},\cdot) \mathrm{d}m_{\Sigma}.
    		\end{split}
    	\end{equation}
    	Since $\lim\limits_{x \to p} u_{i}(x) = u_{i}(p)=0$, we have that
    	\begin{equation}\label{equation_2_lemma_4_proof_of_main_theorem}
    		\lim\limits_{r_{1} \to 0} \int_{\Sigma} u_{i}^{2}(r_{1},\cdot) \mathrm{d}m_{\Sigma} = 0.
    	\end{equation}
    	Recall that we always assume that $A_{i}(r) < \infty$. Thus, we have that
    	\begin{equation}
    		A_{i}(r) = \int_{0}^{r} \int_{\Sigma} \rho \lvert \nabla u_{i} \rvert^{2}(\rho,\xi) \mathrm{d}m_{\Sigma}(\xi) \mathrm{d}\rho \to 0,\quad \text{as } r \to 0.
    	\end{equation}
    	On the other hand, for each $r > 0$, the set
    	\begin{equation}
    		\{\rho \in (\frac{r}{2},r) : \int_{\Sigma} \rho \lvert \nabla u_{i} \rvert^{2}(\rho,\cdot) \mathrm{d}m_{\Sigma} \le \frac{4A_{i}(r)}{\rho}\}
    	\end{equation}
    	has positive $\mathcal{L}^{1}$ measure. In particular, there exists a sequence of $r_{1} \to 0$ with \eqref{equation_1_lemma_4_proof_of_main_theorem} holds such that
    	\begin{equation}\label{equation_3_lemma_4_proof_of_main_theorem}
    		\begin{split}
    			& r_{1} \int_{\Sigma} u_{i}(r_{1},\cdot) \lvert \nabla_{\mathbb{R}^{+}}u_{i}\rvert(r_{1},\cdot) \mathrm{d}m_{\Sigma}\\
    			\le & \lvert u_{i}(r_{1},\cdot) \rvert_{L^{2}(\Sigma,m_{\Sigma})} \Big(\int_{\Sigma} r_{1}^{2} \lvert \nabla u_{i} \rvert^{2}(r_{1},\cdot) \mathrm{d}m_{\Sigma}\Big)^{\frac{1}{2}},\\
    			\to& 0,\quad \text{as } r_{1} \to 0.
    		\end{split}
    	\end{equation}
    	Finally, by Lemma \ref{lemma_3_proof_of_main_theorem}, we have that
    	\begin{equation}\label{equation_4_lemma_4_proof_of_main_theorem}
    		\int_{B_{r_{2}}(p)\setminus B_{r_{1}}(p)} \frac{u_{i}^{2}}{2} \mathrm{d}\mathbf{\Delta} d(x,p)^{2-N} = 0.
    	\end{equation}
    	By \eqref{equation_1_lemma_4_proof_of_main_theorem}, \eqref{equation_2_lemma_4_proof_of_main_theorem}, \eqref{equation_3_lemma_4_proof_of_main_theorem},\eqref{equation_4_lemma_4_proof_of_main_theorem}, we conclude that, for almost all $r > 0$, it holds that
    	\begin{equation}
    		\begin{split}
    			A_{i}(r) & \le \int_{B_{r}(p)} \frac{1}{d(x,p)^{N-2}} \mathrm{d}\mathbf{\Delta} \frac{u_{i}^{2}}{2}\\
    			& = \int_{\Sigma} ru_{i}(r,\cdot) \nabla_{\mathbb{R}^{+}} u_{i}(r,\cdot) + \frac{N-2}{2} u_{i}^{2}(r,\cdot) \mathrm{d}m_{\Sigma}\\
    			& = \int_{\Gamma_{i}(r)} ru_{i}(r,\cdot) \nabla_{\mathbb{R}^{+}} u_{i}(r,\cdot) + \frac{N-2}{2} u_{i}^{2}(r,\cdot) \mathrm{d}m_{\Sigma}.
    		\end{split}
    	\end{equation}
    \end{proof}
    \begin{proof}[Proof of Theorem \ref{main_theorem}]
    	For each $r$, we choose $\beta_{i}(r) \in (0,1]$ be such that
    	\begin{equation}
    		\frac{N-2}{2} = \frac{(1-\beta_{i}(r))\lambda(\Gamma_{i}(r))}{2\sqrt{\beta_{i}(r)\lambda(\Gamma_{i}(r))}}.
    	\end{equation}
    	We let $\alpha_{i}(r) := \sqrt{\beta_{i}(r)\lambda(\Gamma_{i}(r))}$. Then $\alpha_{i}(r)$ is the positive solution to the following equation:
    	\begin{equation}
    		\lambda(\Gamma_{i}(r)) = \alpha_{i}(r)\big(N-2+\alpha_{i}(r)\big).
    	\end{equation}
    	Let $\overline{\Gamma}_{i}(r)$ be a geodesic ball in the sphere $\partial B_{1} \subset \mathbb{R}^{N}$ such that
    	\begin{equation}
    		\frac{m_{\Sigma}(\Gamma_{i}(r))}{m_{\Sigma}(\Sigma)} = \frac{\mathrm{Vol}(\overline{\Gamma}_{i}(r))}{\mathrm{Vol}(\partial B_{1})}.
    	\end{equation}
    	Then, by Theorem \ref{Faber_Krahn}, we get that $\lambda (\Gamma_{i}(r)) \ge \lambda (\overline{\Gamma}_{i}(r))$ and therefore
    	\begin{equation}
    		\alpha_{i}(r) \ge \alpha (\overline{\Gamma}_{i}(r)).
    	\end{equation}
    	By Lemma \ref{Friedland-Haymen_inequality}, we get that
    	\begin{equation}
    		\alpha_{1}(r) + \alpha_{2}(r) \ge \alpha(\overline{\Gamma}_{1}(r))+\alpha(\overline{\Gamma}_{2}(r)) \ge 2.
    	\end{equation}
    	By Lemma \ref{lemma_2_proof_of_main_theorem} and Lemma \ref{lemma_4_proof_of_main_theorem}, we get that
    	\begin{equation}
    		\begin{split}
    			A_{i}^{\prime}(r) & \ge \int_{\Gamma_{i}(r)} 2\sqrt{\beta_{i}(r)\lambda(\Gamma_{i}(r))} u_{i}(r,\xi) \lvert \nabla_{\mathbb{R}^{+}} u_{i} (\cdot,\xi)\rvert(r) + (1-\beta_{i}(r)) \lambda(\Gamma_{i}(r)) r^{-1} \lvert u_{i}(r,\xi) \rvert^{2} \mathrm{d}m_{\Sigma}\\
    			& = 2 \alpha_{i}(r) \int_{\Gamma_{i}(r)} u_{i}(r,\xi) \lvert \nabla_{\mathbb{R}^{+}} u(\cdot,\xi) \rvert(r) + \frac{N-2}{2}r^{-1} \lvert u_{i}(r,\xi)\rvert^{2} \mathrm{d}m_{\Sigma}\\
    			& \ge \frac{2\alpha_{i}(r)}{r}A_{i}(r).
    		\end{split}
    	\end{equation}
    	Finally, by Corollary \ref{corollary_1_proof_of_main_theorem}, we get that
    	\begin{equation}
    		\begin{split}
    			\frac{J^{\prime}(r)}{J(r)} & \ge \frac{A_{1}^{\prime}(r)}{A_{1}(r)} + \frac{A_{2}^{\prime}(r)}{A_{2}(r)} - \frac{4}{r}\\
    			& \ge \frac{2\big(\alpha_{1}(r)+\alpha_{2}(r)\big)-4}{r} \ge 0.
    		\end{split}
    	\end{equation}
    \end{proof}
\section{Proof of Theorem \ref{rigidity_theorem}}
\begin{proof}[Proof of Theorem \ref{rigidity_theorem}]
	If $0<J(r_{1})=J(r_{2})<\infty$, then all inequality in the proof of Theorem \ref{main_theorem} must be equality. At first, for almost all $r \in (r_{1},r_{2})$, it holds that
	\begin{equation}
		\begin{split}
			& \lambda(\Gamma_{i}(r))=\lambda(\overline{\Gamma}_{i}(r))\\
			& \alpha_{1}(r)+\alpha_{2}(r) = \alpha(\overline{\Gamma}_{1}(r)) + \alpha (\overline{\Gamma}_{2}(r)) = 2,
		\end{split}
	\end{equation}
	where $\alpha_{i}(r),\Gamma_{i}(r)$ and $\overline{\Gamma}_{i}(r)$ are as in the proof of Theorem \ref{main_theorem}. By Theorem \ref{Faber_Krahn}, the first equality above implies that $(\Sigma,d_{\Sigma},m_{\Sigma})$ is a spherical suspension. By Theorem \ref{Friedland-Haymen_inequality}, the second equality above implies that $\overline{\Gamma}_{i}(r)$ are half-spheres. If $(\Sigma,d_{\Sigma},m_{\Sigma})$ is an $(N-1)$-dimensional smooth Riemannian manifold, then, by Theorem \ref{Faber_Krahn}, $(\Sigma,d_{\Sigma},m_{\Sigma})$ is isometric to the sphere $\partial B_{1} \subset \mathbb{R}^{N}$. In this case, $\Gamma_{1}(r),\Gamma_{2}(r)$ are upper and lower half-spheres, i.e. there exists a $\nu(r) \in \partial B_{1}$ such that
	\begin{equation}
		\Gamma_{1}(r) = \{x \in \partial B_{1} : x\cdot \nu(r) >0\},\quad \Gamma_{2}(r) = \{x\in \partial B_{1} : x\cdot \nu(r)<0\}.
	\end{equation}
	Next, the inequality in Lemma \ref{lemma_4_proof_of_main_theorem} must be equality. This requires that
	\begin{equation}
		\lvert \nabla u_{i} \rvert^{2} = \mathbf{\Delta} \frac{u_{i}^{2}}{2},\quad \text{in } B_{r_{2}}(p),
	\end{equation}
	i.e. $u_{i}\mathbf{\Delta}u_{i}=0$ in $B_{r_{2}}(p)$. This implies that $u_{i}$ is harmonic in $B_{r_{2}}(p) \cap \{u_{i}>0\}$. On the other hand, the inequality in Lemma \ref{lemma_2_proof_of_main_theorem} must be equality. This requires that
	\begin{equation}
		\int_{\Gamma_{i}(r)} \lvert \nabla_{\Sigma} u_{i}(r,\cdot)\rvert^{2} \mathrm{d}m_{\Sigma} = \lambda(\Gamma_{i}(r)) \int_{\Gamma_{i}(r)} \lvert u_{i}\rvert^{2}(r,\cdot) \mathrm{d}m_{\Sigma},\quad \text{for almost all } r \in (r_{1},r_{2}),
	\end{equation}
	i.e. $u_{i}(r,\cdot)$ is the eigenfunction. Since $\Gamma_{i}(r)$ are half-spheres, we get that $u_{i}$ has the form
	\begin{equation}
		u_{i}(x) = \big(a_{i}(r)\cdot x\big)^{+},\quad \text{for } r_{1}<r=\lvert x\rvert <r_{2},
	\end{equation}
	where $a_{i} : (r_{1},r_{2}) \to \mathbb{R}^{N}$. A simple calculation gives that
	\begin{equation}
		0=\Delta u_{i}(x) = \big(a_{i}^{\prime\prime}(r)+\frac{N+1}{r}a_{i}^{\prime}(r)\big) \cdot x,\quad \text{in } \{r_{1}<\lvert x\rvert <r_{2}\} \cap \{u_{i}>0\}.
	\end{equation}
	This implies that $a_{i}$ solves the following ODE:
	\begin{equation}
		a_{i}^{\prime\prime}(r) + \frac{N+1}{r} a_{i}^{\prime}(r)=0,\quad r_{1}<r<r_{2}.
	\end{equation}
	Therefore, $a_{i}(r) = \nu_{i} + \frac{\sigma_{i}}{r^{N}}$ for some constant vector $\nu_{i},\sigma_{i} \in \mathbb{R}^{N}$. Hence
	\begin{equation}
		u_{i}(x) = \big((\nu_{i}+\frac{\sigma_{i}}{\lvert x \rvert^{N}}) \cdot x\big)^{+},\quad r_{1}<\lvert x \rvert <r_{2}.
	\end{equation}
	By the unique continuation of harmonic function, we get that
	\begin{equation}
		u_{i}(x) = \big((\nu_{i}+\frac{\sigma_{i}}{\lvert x \rvert^{N}}) \cdot x\big)^{+},\quad 0<\lvert x \rvert <r_{2}.
	\end{equation}
	Finally, since $\lim\limits_{x \to p}u_{i}(x)=0$, we get that $\sigma_{i}=0$. While $J(r_{2})>0$, we get that $\nu_{i}$ is non-zero. Thus $u_{i}$ has the form
	\begin{equation}
		u_{1}(x) = k_{1}(x\cdot \nu)^{+},\quad u_{2}(x) = k_{2}(x\cdot \nu)^{-},\quad \text{in } B_{r_{2}}(p),
	\end{equation}
	for some constant $k_{1},k_{2}>0$ and some unit vector $\nu \in \partial B_{1}$.
\end{proof}

\section*{Conflict of interest}
The author declares that there are no conflicts of interest.

\section*{Funding}
No funding was received.

\section*{Ethics approval}
Ethics approval was not required for this research.

\section*{Data availability}
No data was used for the research described in the article.

\bibliographystyle{plain}
\bibliography{biblio}
\end{document}